\documentclass[12pt]{amsart}
\usepackage{amssymb,epsfig,amsmath,latexsym,amsthm}
\usepackage{graphicx}
\usepackage{enumitem}
\usepackage{comment}
\theoremstyle{plain}
\newtheorem{theorem}{Theorem}[section]
\newtheorem{lemma}[theorem]{Lemma}

\newtheorem{proposition}[theorem]{Proposition}
\newtheorem{corollary}[theorem]{Corollary}

\theoremstyle{definition}
\newtheorem{definition}[theorem]{Definition}
\newtheorem{remark}[theorem]{Remark}

\newtheorem{example}[theorem]{Example}

\usepackage{epsfig,color}

\makeatother

\theoremstyle{definition}

\def\fnum{equation}

\numberwithin{equation}{section}
\begin{document}
\title[Whitney sphere along  mean curvature flow]{A note on the evolution of the Whitney sphere along  mean curvature flow}

\author{Celso Viana}

\address{University College London UCL\\Union Building, 25 Gordon Street, London WC1E 6BT}

\email{celso.viana.14@ucl.ac.uk}

\begin{abstract}
We study the evolution of the Whitney sphere along the Lagrangian mean curvature flow. 
We show  that equivariant Lagrangian  spheres  in $\mathbb{C}^n$ satisfying mild geometric assumptions collapse to a point in finite time and the tangent flows converge to a Lagrangian plane with multiplicity two.
\end{abstract}
\maketitle

\section{Introduction}
 The Whitney sphere is  the immersion $F: \mathbb{S}^n \rightarrow \mathbb{R}^{2n}$ given by 
\[F(x_1,\ldots,x_{n+1})=\frac{1}{1+x_{n+1}^2}(x_1,x_1x_{n+1},\ldots,x_n,x_nx_{n+1}).\]
This immersion  is Lagrangian, i.e.,  $F^{*}\omega=0$, where $\omega$ is the standard symplectic form on $\mathbb{R}^{2n}$. From the point of view of topology, the Whitney sphere is interesting since  it has the best topological behavior, namely it fails to be embedded   only at the north and south pole where it has a transversal double point. An well known result of Gromov asserts that there are no embedded Lagrangian spheres in $\mathbb{C}^n$. On the geometry side, this immersion  can be characterized by many geometric rigidity properties, see \cite{CU,RU}. In this sense, the Whitney sphere plays the role of totally umbilical hypersurfaces in $\mathbb{R}^n$ in the class of Lagrangian submanifolds.

Another interesting aspect of the Whitney sphere is that it appears  as  a limit surface under Lagrangian mean curvature flow of some well-behaved Lagrangian submanifolds in $\mathbb{R}^4$. Recall that the mean curvature flow (MCF) of an immersion $F_0: M^k\rightarrow \mathbb{R}^m$ is a map $F: M\rightarrow [0,T]\rightarrow \mathbb{R}^m$ such that $F(x,0)=F_0$ and   satisfies the equation
\[
\frac{d}{dt} F= H,
\]
where $H$ is the mean curvature vector of $M^n$. It was shown by K. Smoczyk  that the Lagrangian condition is preserved by MCF when the ambient space is a  K\"{a}hler-Einstein manifold. 
The Lagrangian mean curvature flow gained a lot of interest recently as a potential tool to find minimal Lagrangian (special Lagrangian)  in a given homology class or Hamiltonian isotopy class of a Calabi-Yau manifold. Special Lagrangian submanifolds  have the remarkable property of being area minimizing by means of  calibration arguments. The  classical approach of minimizing area in a given class, however, does not seem very effective to find smooth special Lagrangian as shown by Schoen and Wolfson in \cite{SW}. 

Ideally, one could hope that the evolution of  well behaved Lagrangian submanifolds along mean curvature flow  to converge to  special Lagrangians. In a series of works, A. Neves  showed that finite time singularities are unavoidable in the Lagrangian mean curvature flow in general, see \cite{N,N2}. 
It is  constructed in \cite{N} a non-compact zero Maslov class Lagrangian in $\mathbb{R}^4$ with bounded Lagrangian angle and in the
same Hamiltonian isotopy class of a Lagrangian plane that nevertheless develops a singularity  in finite time. At the singular time the limit
surface pictures like a connect sum of a smooth Lagrangian (diffeomorphic to a Lagrangian plane) with a Whitney Sphere. Such construction were later generalized to $4$-dimensional Calabi-Yau manifolds, see \cite{N2}.

There are very few results regarding the evolution of compact Lagrangian submanifolds in $\mathbb{C}^n$. Motivate by this, we investigate the evolution of the Whitney sphere along mean curvature flow. Despite its  many geometric  properties, it is not a self-similar solution of the flow.  
By exploiting its rotationally symmetries,  one can reduce its  mean curvature flow to a flow   about curves in  the plane. As a particular case of our main result we prove
\newline

\noindent	\textit{Let $F: \mathbb{S}^n\times [0,T)\rightarrow \mathbb{C}^{n}$ be the  maximal existence mean curvature flow of the Whitney sphere. Then $F_T(x)=\{0\}$ for every $x\in \mathbb{S}^n$. The tangent flow at the origin is a Lagrangian plane with multiplicity two.}
\newline

A Lagrangian submanifold $L\subset\mathbb{C}^{n}$ is called \textit{equivariant} if there exists a antipodal invariant curve $\gamma: I\rightarrow \mathbb{C}$ such that $L$ can be written as
\begin{equation*}\label{equivariant1}
L=\{(\gamma(u)\, G_1(x),\ldots,
\gamma(u)\, G_n(x))\in \mathbb{C}^n\,:\, G: \mathbb{S}^{n-1}\rightarrow \mathbb{R}^{n}\},
\end{equation*}
where $G$ is a the standard embedding of $\mathbb{S}^{n-1}$ in $\mathbb{R}^n$.
Using spherical coordinates on $\mathbb{S}^n$, $(\cos(u)\,G(x), \sin(u))$, we check that the  Whitney sphere is equivariant with associated curve $\gamma_0: (0,2\pi)\rightarrow \mathbb{R}^2$  given by:
\[\gamma_0(u)=\bigg(\frac{\sin(u)}{1+\cos^2(u)},\frac{\sin(u)\cos(u)}{1+\cos^2(u)}\bigg).\]
The equivariant property is preserved by the mean curvature flow and the corresponding evolution equation for $\gamma_t$ is
\begin{equation}\label{equivariant flow}
\frac{d\gamma}{dt}=\overrightarrow{k}-(n-1)\, \frac{\gamma^{\perp}}{|\gamma|^2}.
\end{equation}
Here$\overrightarrow{k}$ denotes the curvature vector of $\gamma$, it is  defined by $\overrightarrow{k}= \frac{1}{|\gamma^{\prime}|}\frac{d}{du}\frac{\gamma^{\prime}}{|\gamma^{\prime}|}$, and  $\gamma^{\perp}$ denotes the normal projection of the position vector $\gamma$. This flow is known as the \textit{equivariant flow}.

 \begin{definition}
	Let $\mathcal{C}$ be the  set of antipodal invariant figure eight curves
	$\gamma: \mathbb{S}^1\rightarrow \mathbb{C}$ with only one   self-intersection  which is transversal and  located at the origin.
\end{definition}
\begin{definition}
Let $\Omega_{\alpha}$ be the antipodal invariant region in $\mathbb{R}^2$ bounded by two lines through the origin with  angle between them equal to  $\alpha$.
\end{definition}
\begin{figure}[h]\label{Whitney Sphere}
	\includegraphics[scale=0.13]{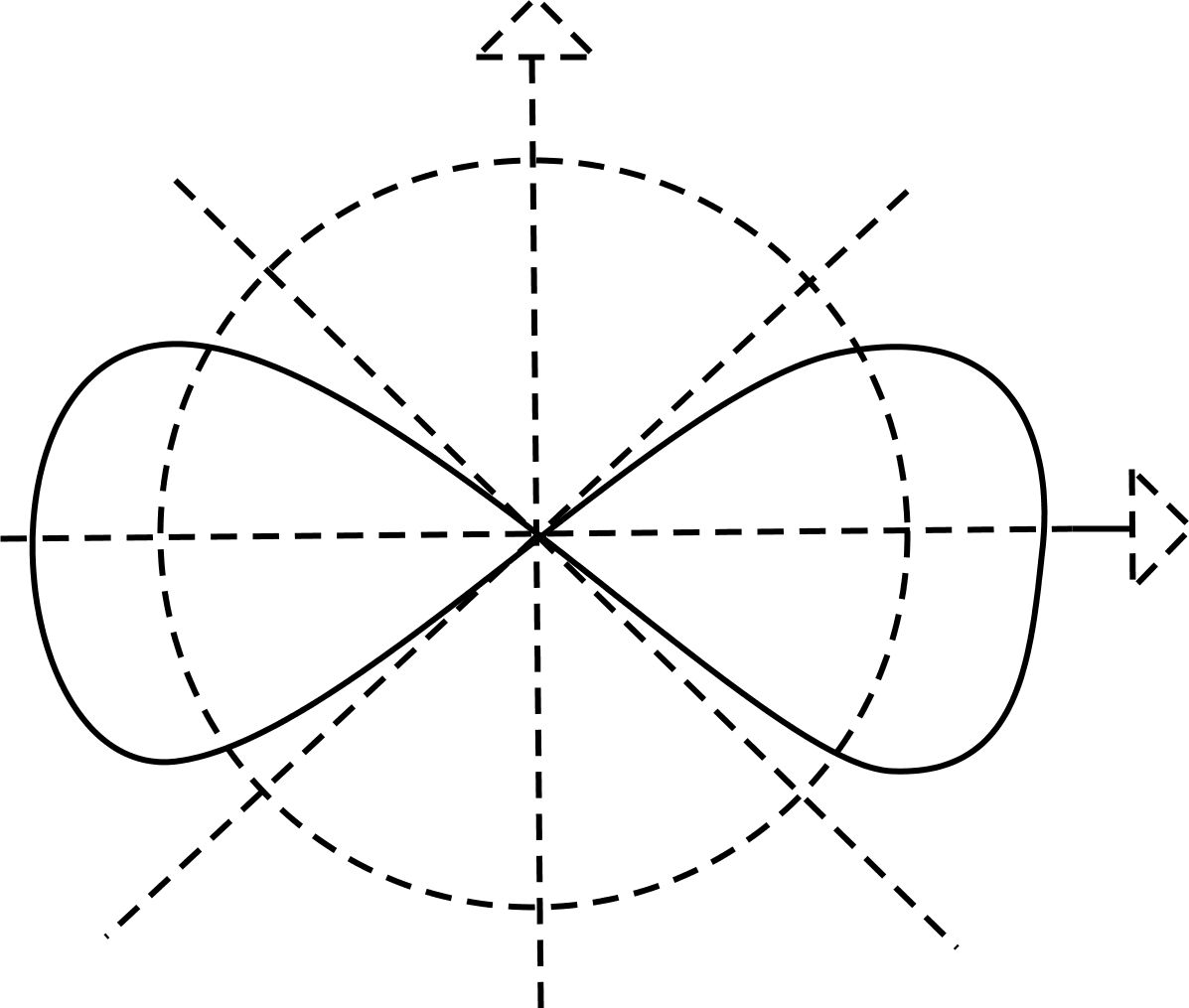}
	\caption{Whitney Sphere}
\end{figure}

\begin{theorem}\label{main theorem}
\textit{Let  $\gamma$ be a curve in $\mathcal{C}$ satisfying at least one of the following assumptions:  
	\begin{itemize}
		\item[i)] $\{\gamma\}\cap \mathbb{S}^1(R)$ has at most $4$ points for every $R>0$;
		\item[ii)]  $\{\gamma\}\subset \Omega_{\frac{\pi}{n}}$.
	\end{itemize}
	   	If $\{\gamma_t\}_{t\in [0,T)}$ is the maximal equivariant flow of $\gamma$, then $\gamma_T=\{0\}$. Moreover, the tangent flow at the origin is a line with multiplicity two.}
\end{theorem}

\begin{remark}
	The assumptions in Theorem \ref{main theorem} are sharp. In Section 3 we construct for every $\alpha>\frac{\pi}{2}$ a curve $\gamma\in \mathcal{C}$ and $\{\gamma\}\subset \Omega_{\alpha}$   that  develops a non-trivial singularity along the flow (\ref{equivariant flow}) at the origin when $n=2$. 
\end{remark}

The proof of Theorem \ref{main theorem} follows closely the ideas in \cite{N,N1} where it is shown that singularities for the mean curvature flow of  monotone Lagrangian submanifolds in $\mathbb{R}^4$ are modeled on area minimizing cones.

\section{Preliminaries}

Let $L$ be a Lagrangian submanifold in $\mathbb{C}^n$. This implies that $\omega\lvert_{L}=0$, where $\omega= \sum_{i=1}^{n}\frac{\sqrt{-1}}{2}dz_i\wedge \overline{dz_i}$ is the standard symplectic form on $\mathbb{C}^n$. Let $\Omega$ be the complex valued $n$-form given by 
$$\Omega= dz_1\wedge\ldots\wedge dz_n.$$
A standard computation implies that
\begin{eqnarray}\label{calibration}
	\Omega\lvert_L= e^{i\theta}vol_L.
\end{eqnarray}
The multivalued function $\theta$ is called the \textit{Lagrangian angle} of $L$. If $\theta$ is a single valued function, then $L$ is called \textit{zero-Maslov class}. If $\theta=\theta_0$, then $L$ is calibrated by $\text{Re}(e^{-i\theta_0}\Omega)$  and hence area-minimizing. In this case, $L$ is called  \textit{special Lagrangian}. More generally, the Lagrangian angle and the geometry of $L$ are related through $\overrightarrow{H}= J(\nabla \theta)$. Recall also the Liouville one form given by 
\[
\lambda=\sum_{i=1}^{n}x_idy_i-y_idx_i.
\]
One can check that $d\lambda=\omega$. In particular, $[\lambda]\in H_1(L)$. When $[\lambda]= c[d\theta]$ for some $c\in\mathbb{R}$, then $L$ is said to be a \textit{monotone Lagrangian}.

Let $L$ be a equivariant Lagrangian submanifold in $\mathbb{R}^{2n}$. Hence, there exists a regular curve $\gamma$ in $\mathbb{R}^2$ such that 
\begin{eqnarray}\label{equivariant}
L=\{(\gamma\, G_1,\ldots,
\gamma\, G_n)\in \mathbb{R}^{2n},\, \sum_{i=1}^{n}G_i^2=1.\}
\end{eqnarray}
After choosing a parametrization of $\gamma$ we have
\begin{equation}\label{eq2}
\Omega_L:= dz_1\wedge\cdots\wedge dz_n\bigg|_{L}= e^{i\theta}\text{vol}_L=\frac{\gamma^{\prime}}{|\gamma^{\prime}|}\cdot\bigg(\frac{\gamma}{|\gamma|}\bigg)^{n-1} \text{vol}_L,
\end{equation}
where $z\cdot w$ denotes the standard multiplication of complex numbers; here we consider $\gamma$ as complex valued function. The Lagrangian angle relates to the geometry of L.

If $L_t$ is the mean curvature flow  starting at $L$, then $L_t$ shares the same rotational symmetries of $L$, i.e., 
$L_t=\{\gamma_t\cos(\alpha),
\gamma_t\sin(\alpha)), \alpha \in \mathbb{R}/2\pi \mathbb{Z}\}$. Moreover,
\begin{equation}
\frac{d\gamma}{dt}=\overrightarrow{k}-(n-1) \frac{\gamma^{\perp}}{|\gamma|^2}.
\end{equation}
Although the term $\frac{\gamma^{\perp}}{|\gamma|^2}$ is not well defined at the origin the quantity has its meaning even when
a curve goes through the origin as we can see below.

\begin{lemma}
	Let $\gamma: [-a,a]\rightarrow \mathbb{R}^2$ a smooth regular curve such that $\gamma(0)=0$. Then
	$$\lim_{s\rightarrow 0}\frac{\gamma^{\perp}}{|\gamma|^2}(s)=\frac{1}{2}\overrightarrow{k}(0).$$
\end{lemma}
\begin{proof}
	Let us write the left hand side as
	\begin{eqnarray*}\frac{\gamma^{\perp}}{|\gamma|^2}(s)= 
		\frac{1}{|\gamma|^2}\left< \gamma,
	i\frac{\gamma^{\prime}}{|\gamma^{\prime}|}\right> i\frac{\gamma^{\prime}}{|\gamma^{\prime}|}=  \frac{s^2}{|\gamma|^2}\left< \frac{\gamma - s\gamma^{\prime}(0)}{s^2}, i\frac{\gamma^{\prime}}{|\gamma^{\prime}|}\right> i\frac{\gamma^{\prime}}{|\gamma^{\prime}|}.
	\end{eqnarray*}
Using that $\lim_{s \rightarrow 0}
	\frac{\gamma(s)}{s}=\gamma^{\prime}(0)$ and applying the L'Hopital's rule twice, we obtain
		\begin{eqnarray*}
\lim_{s\rightarrow 0}\frac{\gamma^{\perp}}{|\gamma|^2}(s)=  \frac{1}{2}\frac{1}{|\gamma^{\prime}(0)|^2}\left< \gamma^{\prime\prime}(0),i\frac{\gamma^{\prime}(0)}{|\gamma^{\prime}|(0)}\right> i\frac{\gamma^{\prime}(0)}{|\gamma^{\prime}(0)|}= \frac{1}{2}\overrightarrow{k}(0).
	\end{eqnarray*}  
\end{proof}

\begin{proposition}[Neves \cite{N2}]\label{non-avoidance equivariant flow}
Let $\gamma_{i,t}:[-a,a]\rightarrow \mathbb{R}^2$,  $i=1,2$ and $0\leq t \leq T$, smooth regular curves satisfying
\begin{enumerate}
  \item $\gamma_{i,t}(-s)=-\gamma_{i,t}(s)$ for all $0\leq t \leq T$ and for every $s \in [-a,a]$ .
  \item The curve $\gamma_{i,t}$, $i=1,2$, solves the equation
$$\frac{d\gamma}{dt}=\overrightarrow{k}-(n-1)\frac{\gamma^{\perp}}{|\gamma|^2}.$$
  \item $\gamma_{1,0}\cap \gamma_{2,0}=\{0\}$ (non-tangential intersection) and
$\partial \gamma_{1,t}\cap \gamma_{2,t}=\partial \gamma_{2,t}\cap \gamma_{1,t}=\emptyset$ for all $t$.
\end{enumerate}
Then for all $0\leq t\leq T$ we have $\gamma_{1,t}\cap\gamma_{2,t}=\{0\}$.
\end{proposition}
\begin{proof}
It suffices to restrict to what happens near the origin since the proposition follows
from the standard maximum principle applied to the first time of tangential intersection. 

First notice that $\gamma_{i,t}$  can be written as a graph on $[-\delta,\delta]$ for some $\delta>0$. Hence,
$\gamma_{i,t}(s)=(s,f_{i,t}(s))$ and we define $\alpha_{i,t}(s)=\frac{f_{i,t}(s)}{s}$. Let's check that $\alpha_{i,t}(s)$ is smooth: if $s\neq 0$, then
\begin{eqnarray}\label{derivadas}
\alpha'(s)= \frac{f's-f}{s^2}\quad \text{and}\quad \alpha''(s)&=&\frac{(f''s+f'-f')s^2-(f's-f)2s}{s^4} \nonumber\\ &=&\frac{f''}{s} + 2\frac{f-f's}{s^3}.
\end{eqnarray}
Since $f(0)=0$ and $f''(0)=0$ (item (1)), we can apply L'Hopital's rule to show that $\alpha'$ and $\alpha''$ in (\ref{derivadas}) have a limit when $s\rightarrow 0$. Hence, $\alpha$ is twice differentiable.

Finally we consider the function $u_t(s)=\alpha_{1,t}-\alpha_{2,t}$. Notice that  $u_0>0 $ by assumption
$(3)$  and $u_t(s)=u_t(-s)$.  Recall that  in the case of a graph $\gamma(s)=(s,f(s))$ we have
\[\gamma^{\prime}=(1,f^{\prime}), \quad \nu=\frac{(f^{\prime},-1)}{\sqrt{1+(f^{\prime})^2}}\quad \textrm{and}\quad
\overrightarrow{k}=-\frac{f^{\prime \prime}}{(1+(f^{\prime})^2)^{\frac{3}{2}}}\nu.
\]
Besides, \[\frac{\gamma^{\perp}}{|\gamma|^2}=\frac{sf^{\prime}-f}{s^2+f^2}\frac{1}{\sqrt{1+(f^{\prime})^2}}\nu.\]
Therefore, the equation $\frac{d\gamma}{dt}^{\perp}=\overrightarrow{k}-\frac{z^{\perp}}{|z|^2}$ implies
$$\frac{df}{dt}=\frac{f^{\prime \prime}}{1+(f^{\prime})^2}+(n-1)\big(\arctan \frac{f}{s}\big)^{\prime}.$$
Standard  computations imply that $\alpha_{i,t}=\frac{f_{i,t}}{s}$ satisfies
\[\frac{d\alpha_{i,t}}{dt}=\frac{\alpha_{i,t}^{\prime \prime}}{1+(s\alpha_{i,t}^{\prime}+\alpha_{i,t})^2}
+ \frac{\alpha_{i,t}^{\prime}}{s}\frac{2}{1+(s\alpha_{i,t}^{\prime}+\alpha_{i,t})^2}
+ (n-1)\frac{\alpha_{i,t}^{\prime}}{s}\frac{1}{1+\alpha_{i,t}^2}.
\]
Now we proceed to find the equation
for $\frac{du_t}{dt}$. Using that  $\frac{\alpha_{i,t}}{s}$ is also smooth, one can  checked that
$$\frac{du_t}{dt}=C_1^2 u_t^{\prime \prime}+C_2 u_t^{\prime} + C_3 u_t +C_4^2 \frac{u_t^\prime}{s},$$
where each  $C_k$ is a smooth and bounded function.
By item (3), the function $u_{t=0}$ is strictly positive since $\gamma_1$ and $\gamma_2$ have a non-tangential intersection at the origin. 

 Suppose $T_1$ is the first time where $u_t$ has a zero  say at $s_0$. Hence, $s_0$ is a minimum point as $u_{T_1}\geq0$. We consider the function
$v_t=u_t e^{-Ct}+\varepsilon(t-T_1)$ where $C$ is very large and $\varepsilon$ is a very small positive number. So at $(s_0,T_1)$ we have
\[0\geq \frac{dv_t}{dt}(s_0,T_1)=\frac{du_t}{dt}(s_0,T_1)e^{-CT_1}+\varepsilon\geq \varepsilon + C_4^2 \frac{u_t^{\prime}(s_0)}{s_0}e^{-CT_1}.\]
We  used in the equality part that $u_{T_1}(s_)=0$ and that $u_{T_1}^{\prime}(s_0)=0$ and $u_{T_1}^{\prime \prime}(s_)\leq 0$ since $s_0$ is a minimum point for $u_{T_1}$.
If $s_0\neq 0$ then the second term in the right hand side is zero and we get a contradiction. If
$s_0=0$ then  that term is just $u_t^{\prime \prime}(0)e^{-CT_1}$ by the L'Hopital's rule, hence,  non-negative and we obtain
a contradiction again.
\end{proof}

\begin{corollary}\label{preserved C}
	\textit{The set $\mathcal{C}$ is preserved by the equivariant flow.}
\end{corollary}
\begin{proof}
	The symmetries of the curve $\gamma$ are preserved by the equivariant flow, hence $\gamma_t$ is also antipodal invariant.    Proposition \ref{non-avoidance equivariant flow} guarantees that the only self intersection of $\gamma_t$ is at the origin.
	Moreover, Proposition \ref{non-avoidance equivariant flow} also implies that $\gamma_t(s)$ non trivially intersect the  line $s\overrightarrow{v}_{s\in \mathbb{R}}$  only in at most one pair of antipodal points   for any $t \in [0,T)$.  On the other hand, by Proposition 1.2 in \cite{A},  this intersection is never tangential unless it is trivial, i.e., is at the origin.  Therefore, if $\gamma\in \Omega_{\frac{\pi}{2}}$, then so is $\gamma_t$. Finally, by Theorem 1.3 in \cite{A}, the number of intersections between $\gamma$ and $\mathbb{S}^1(R)$ is non-increasing along flow.
\end{proof}

\begin{lemma}\label{strictly contained cone}
	If $\gamma \in \Omega_{\frac{\pi}{n}}$, then for every $t>0$ there exists $\delta_t>0$ such that $\{\gamma_t\}\subset\Omega_{\frac{\pi}{n}-\delta_t}$. \end{lemma}
\begin{proof}
Since $\gamma \in \mathcal{C}$  is antipodal invariant and passes through the origin, one can check that $\lim_{s \rightarrow 0} \frac{\gamma^{\perp}}{|\gamma|^2}(s)=0$, where $\gamma(s)$ is a local parametrization of $\gamma$ with $\gamma(s)=-\gamma(-s)$. By Lemma 2.1, we have that $\overrightarrow{k}(z_0)=\overrightarrow{k}(-z_0)=0$, where $\gamma(z_0)=\gamma(-z_0)=0$. Consequently, $\overrightarrow{H}(z_0)=\overrightarrow{H}(-z_0)=0$. This implies that $z_0$ and $-z_0$ are critical points of the Lagrangian angle $\theta_L$. It can be check easily that they correspond to local minimum and local maximum critical points.  The strong maximum principle applied to $\frac{d}{dt}\theta= \Delta \theta$ implies that $\theta_t(z_0)< \theta(z_0)$ and $\theta_t(-z_0)> \theta(-z_0)$.
\end{proof}

Let us use  $\text{Area}(\gamma)$ to denote the area enclosed  by $\gamma\in \mathcal{C}$.
By the Stokes' theorem we have that
$\textrm{Area}(\gamma_t)=-\frac{1}{2}\int_{\gamma_t} \langle \gamma_t,\nu\rangle d_{\gamma_t}$, where $\nu$ is the unit outward normal vector of $\gamma$.
\begin{lemma} \label{variation area}
	\begin{eqnarray*}
	\pi(T-t)\leq \textrm{Area}(\gamma_t)-\text{Area}(\gamma_T)\leq 3\pi(T-t).
\end{eqnarray*}
\end{lemma}
\begin{proof}
	Let $\gamma_t(u)$ be a parametrization of $\gamma_t$. Using  that $\nu=i \frac{\gamma_t'}{|\gamma_t'|}$, we have   that $\textrm{Area}(\gamma_t)=-\frac{1}{2}\int_{\gamma_t} \langle \gamma_t,i\,\gamma_t^{\prime}\rangle du$. Hence,
	\begin{eqnarray*}
		\textrm{Area}^{\prime}(t)&=&-\frac{1}{2}\int_{\gamma_t} \bigg(\langle \partial_t \gamma, i\,\gamma_t^{\prime}\rangle  + \langle
		\gamma, i\,(\partial_t\gamma)^{\prime} \rangle\bigg) du \\
		&=&-\frac{1}{2}\int_{\gamma_t}\bigg(\langle \partial_t \gamma, i\,\gamma_t^{\prime}\rangle + \langle \gamma_t,i\partial_t\gamma\rangle^{\prime} - \langle \gamma^{\prime},i\,\partial_t\gamma\rangle\bigg)du \\
		&=&-\int_{\gamma_t}\langle \partial_t \gamma, i\,\gamma_t^{\prime}\rangle\, du -\frac{1}{2}\int_{\gamma_t}\langle \gamma_t,i\partial_t\gamma\rangle^{\prime}\,du = -\int_{\gamma_t}\langle \partial_t \gamma, \nu\rangle\, d\gamma_t.
	\end{eqnarray*}
	The last equality follows from the Fundamental Theorem of Calculus.
	Hence,
	\begin{eqnarray*}
		\text{Area}'(t)=-\int_{\gamma_t}\bigg< \overrightarrow{k} -(n-1)\frac{z^{\perp}}{|z|^2}, \nu\bigg>\, d_{\gamma_t}=-\int_{\gamma_t}\langle \overrightarrow{k}, \nu\rangle d_{\gamma_t}.
	\end{eqnarray*}
The last equality follows from the Divergence Theorem applied to vector field $X= \frac{z}{|z|^2}$	and the fact that $z=0$ is  not in the interior of the region enclosed by $\gamma_t$.
		Combining the  Gauss-
	Bonnet theorem  and the fact  that the exterior angle $\alpha_t$  of $\gamma_t$ at the origin is in $[-\pi,\pi]$ we obtain
	\[\int_{\gamma_t}\langle \overrightarrow{k},-\nu\rangle d_{\gamma_t} +
	\alpha_t=2\pi\Longrightarrow \pi \leq \int_{\gamma_t}\langle \overrightarrow{k},-\nu\rangle d_{\gamma_t} \leq 3\pi.\]
	Therefore, $-3\pi \leq \textrm{Area}^{\prime}(\gamma_t)\leq -\pi$. The Lemma now follows if we integrate this quantity from $t$ to $T$.
\end{proof}

\section{Proof of the Theorem}
Let $L_t$ be a solution of the mean curvature flow starting on a $k$-dimensional submanifold $L$ in $\mathbb{R}^m$. Consider the backward heat kernel
\[\Phi_{x_0,T}(x,t)= \frac{1}{(4\pi(T-t))^{\frac{k}{2}}}e^{-\frac{|x-x_0|^2}{4(T-t)}}.\]
The following formula is known as the Huisken's monotonicity formula:
\begin{eqnarray}\label{huisken monotonicity}
	\frac{d}{dt}\int_{L_t} f_t \Phi_{x_0,T}d\mathcal{H}^k &=&  \nonumber \\
	&&\int_{L_t} \bigg(\frac{d}{dt}f_t - \Delta f_t - \bigg|H- \frac{(x-x_0)^{\perp}}{2(T-t)}\bigg|^2f_t 
	\bigg)\Phi_{x_0,T} d\mathcal{H}^k,
\end{eqnarray}
where $d\mathcal{H}^k$ denotes the $k$-dimensional Hausdorff measure.
 
Recall that if $\{L_t\}_{t\in[0,T)}$ is the Lagrangian mean curvature flow starting at $L$, then \[L_s^{\sigma}=\sigma\,(L_{T+\frac{s}{\sigma^2}}-x_0),\] for $s \in [-T\lambda^2,0)$, also satisfies the Lagrangian mean curvature flow and is referred as the tangent flow at $x_0$. The following is a restatement of Theorem \ref{main theorem}:

\begin{theorem}
	\textit{Let $\gamma$ be a curve in $\mathcal{C}$  which satisfies  at least one of the following assumptions
	\begin{itemize}
	\item[i)] $\{\gamma\}\cap \mathbb{S}^1(R)$ has at most $4$ points for every $R>0$;
	\item[ii)]  $\{\gamma\}\subset \Omega_{\frac{\pi}{n}}$.
\end{itemize}
If $\{\gamma_t\}_{t\in [0,T)}$ is the maximal equivariant flow of $\gamma$, then $\gamma_T=\{0\}$. Moreover, the tangent flow at the origin is a line with multiplicity two.}
\end{theorem}
\begin{proof}
Let us  prove first that if $z=0$ is a singular point, then $\gamma_T=\{0\}$. Arguing by contradiction, we assume that $z=0$ is a singular point for $\{\gamma_t\}_{0\leq t<T}$ and  $\gamma_T\neq \{0\}$. 
Given $\sigma_i\rightarrow\infty$, let $\gamma_s^{i}=\sigma_i \gamma_{T+\frac{s}{\sigma_i^2}}$.
\begin{lemma}Let $a$ and $b$ real numbers such that $a<b<0$. Then
\[\lim_{i\rightarrow\infty}\int_{a}^{b}\int_{\gamma_s^i\cap A(\frac{1}{\eta},\eta,0)} \bigg(|\overrightarrow{k}|^2+|\gamma^{\perp}|^2\bigg)\,d\mathcal{H}^1 ds=0,\]
where $A(\frac{1}{\eta},\eta,0)$ is an annulus centered at $z=0$ with inner and outer radius $\eta$ and $\frac{1}{\eta}$, respectively.
\end{lemma}
\begin{proof}
	Let $L_s^i$ be the  immersed Lagrangian sphere in $\mathbb{C}^2$ obtained via $L_s^i=(\gamma_s^iG_1,\ldots,\gamma_s^iG_n)$.
 It is proved in Lemma 5.4 in \cite{N} that
\begin{eqnarray}\label{neves equation}
\lim_{i\rightarrow\infty}\int_{a}^{b}\int_{L_s^i\cap B_R(0)} \bigg(|H|^2+|x^{\perp}|^2\bigg)\, d\mathcal{H}^n(x)ds=0,
\end{eqnarray}
where $H$ is the mean curvature vector of $L_s^i$.
For the convenience of the reader let us recall the proof of this fact. It is a standard computation to check that the Lagrangian angle $\theta$ obeys the following evolution equation $\frac{d}{dt}\theta_{i,s}^2= \Delta \theta_{i,s}^2 - 2|H|^2$. Applying (\ref{huisken monotonicity}) with $f_t= \theta_{i,s}^2$ and $f_t=1$, we obtain
\begin{eqnarray}\label{equation1}
\frac{d}{ds} \int_{L_s^i} \theta_{i,s}^2 \Phi d\mathcal{H}^n&=& \int_{L_s^i}\bigg(-2 |H|^2 - \bigg|H- \frac{x^{\perp}}{2s}\bigg|^2 \theta_{i,s}^2\bigg) \Phi\, d\mathcal{H}^n \\
\frac{d}{ds} \int_{L_s^i} \Phi d\mathcal{H}^n&=& \int_{L_s^i} - \bigg|H- \frac{x^{\perp}}{2s}\bigg|^2\Phi\, d\mathcal{H}^n,
\end{eqnarray}
respectively. Integrating (\ref{equation1}) from $a$ to $b$ gives
\begin{eqnarray*}
	2\lim_{i  \rightarrow \infty}\int_{a}^{b}\int_{L_s^i} |H|^2\, \Phi\, d\mathcal{H}^n ds &\leq& \lim_{i\rightarrow\infty}\int_{L_b^i} \theta_{i,b}^2 \Phi\,d\mathcal{H}^n - \lim_{i\rightarrow\infty}\int_{L_a^i} \theta_{i,a}^2 \Phi\,d\mathcal{H}^n	=0.
\end{eqnarray*}
The last inequality follows from the scale invariance and monotonicity of $\int_{L_t}\theta^2\, \Phi d\mathcal{H}^n$. Similarly, we obtain
\begin{eqnarray*}
	\lim_{i  \rightarrow \infty}\int_{a}^{b}\int_{L_s^i} \bigg|H- \frac{x^{\perp}}{2s}\bigg|^2\, \Phi\, d\mathcal{H}^n ds=
	\lim_{i\rightarrow\infty}\int_{L_b^i} \Phi\,d\mathcal{H}^2 - \lim_{i\rightarrow\infty}\int_{L_a^i} \Phi\,d\mathcal{H}^n=0.
\end{eqnarray*}
It follows from the triangular inequality that 
\begin{eqnarray*}
	\lim_{i  \rightarrow \infty}\int_{a}^{b}\int_{L_s^i} \bigg|\frac{x^{\perp}}{2s}\bigg|^2\, \Phi\, d\mathcal{H}^n ds=0.
\end{eqnarray*}
This completes the proof of (\ref{neves equation}).
As $|H|^2=|\overrightarrow{k}-(n-1)\frac{\gamma^{\perp}}{|\gamma|^2}|^2$ and $|x^{\perp}|^2=|\gamma^{\perp}|^2$, we obtain for each $\eta>0$ that
\[\lim_{i\rightarrow\infty}\int_{a}^{b}\int_{\gamma_s^i\cap A(\frac{1}{\eta},\eta,0)} \bigg(|\overrightarrow{k}|^2+|\gamma^{\perp}|^2\bigg)d\mathcal{H}^1 ds=0.\]
\end{proof}
  From previous lemma it follows that for almost every $s\in (a,b)$ that
  \[
  \lim_{i\rightarrow\infty}\int_{\gamma_s^i\cap A(\frac{1}{\eta},\eta,0)} \bigg(|\overrightarrow{k}|^2+|\gamma^{\perp}|^2\bigg)d\mathcal{H}^1=0.
  \]
  This implies that   $\gamma_s^{i}$ converges to a union of  lines in $C_{loc}^{1,\frac{1}{2}}(\mathbb{R}^2-\{0\})$. In fact, each connected component of $\gamma_s^i$ inside $B_R(0)-\{0\}$  converge to a line segment with multiplicity one since the convergence is in $C_{loc}^{1,\frac{1}{2}}(\mathbb{R}^2-\{0\})$.

Assume first that $\gamma$ satisfies \textbf{item i)}, then by Proposition \ref{non-avoidance equivariant flow} and Corollary \ref{preserved C}, the curve   $\gamma_s^{i}$ in $B_R(0)-\{0\}$ has two embedded connected components. Hence, each converges to a line segment  with multiplicity one in $B_R(0)-\{0\}$. Equivalently, in a neighborhood of the origin   $L_t$ is a union of two smooth embedded discs intersecting transversally at a interior point. Hence, each piece of $L_s^{i}$ converges weakly to a plane with multiplicity one. Since  $\gamma_T\neq \{0\}$, we can talk about the localized  Gaussian density of each connected component of  $L_t\cap B_r(0)$ computed at $(0,T)$ which will be very close to one. Applying White's Local Regularity Theorem, see localized version Theorem 5.6 in \cite{E}),to each component of $L_t\cap B_r(0)$, we conclude that the  origin is not a singularity of $\{L_t\}_{t\in[0,T)}$, contradiction.

To handle other connects components of $\gamma_s^i$ in $B_{4R}(0)$ we  study the Lagrangian angle $\theta_s^i$. Let $\beta$ be a primitive of $\lambda_L$. It is proved in \cite{N1} that  $\nabla \beta= J(x^{\perp})$ and  $\frac{d}{dt}\beta= \Delta \beta - 2\theta$. This implies that the function $u= \beta + 2(t-t_0)\theta$ satisfies $\frac{d}{dt}f(u)= \Delta f(u) - f^{\prime\prime}(u)|x^{\perp}+ 2(t-t_0)H|^2$, where $f\in C_0^{\infty}(\mathbb{R})$. Plugging the function $f(u)$ in (\ref{huisken monotonicity}), we obtain
\begin{eqnarray*}
	\frac{d}{ds}\int_{L_s^i} f(u_s^i) \Phi = - \int_{L_s^i} \bigg|H- \frac{x^{\perp}}{2s} \bigg|^2f(u_s^i) \, \Phi\,  + f^{\prime\prime}(u_s^i)\bigg|x^{\perp}+ 2(s-s_0)H\bigg|^2\Phi.
\end{eqnarray*}
Integrating this formula from  $-1$ to $s_0$ and using (\ref{neves equation}), we obtain
\begin{eqnarray*}
	\lim_{i  \rightarrow \infty} \int_{L_{s_0}^i\cap B_{4R}(0)} f(\beta_{s_0}^i)\Phi\, =\, \lim_{i  \rightarrow \infty} \int_{L_{-1}^i\cap B_{4R}(0)} f(\beta_{-1}^i - 2(1 + s_0)\theta_{-1}^i)\Phi.
\end{eqnarray*}
Let $\gamma^i$ be a connected component of $\gamma_s^i$ in $B_{4R}(0)$ that intersects $B_R(0)$ and does not passes through the origin. Since $|\nabla f(\beta_s^i)|$ is bounded, there exists a constant $b_{s_0}$ such that $\lim_{i  \rightarrow \infty} f(\beta_{s_0}^i)=f(b_{s_0})$. Similarly, $\lim_{i  \rightarrow \infty} f(\beta_{-1}^i)=f(b_{-1})$. As before, $\gamma_i$ converges in $C^{1,\frac{1}{2}}(\mathbb{R}^2-\{0\})$ to lines $l_{\overrightarrow{v_1^s}}$ and $l_{\overrightarrow{v_2^s}}$ in the direction of  the vectors $\overrightarrow{v_i^s}$. Moreover, 
\begin{eqnarray*}
\lim_{i  \rightarrow \infty} \int_{\gamma^i} f(\beta_{-1}^i - 2(1 + s_0)\theta_{-1}^i)\Phi \,d\mathcal{H}^1= \sum_{i=1}^{2}\int_{l_{\overrightarrow{v_i}}} f(b_{-1}- 2(1+s_0)\theta_i)\Phi\,d\mathcal{H}^1.
\end{eqnarray*}
Note that (\ref{eq2}) implies that $\theta_s^i$ converge to a constant in each connected component of $\gamma_s^i\cap (B_R(0)-B_r(0))$.
We claim that $\theta_1=\theta_2$. Otherwise, by choosing  $f$  with  support near $b_{s_0}$ and equal to $1$ near $b_{s_0}$, we obtain
\[
\sum_{i=1}^{2}\int_{l_{\overrightarrow{v_i}}\cap B_R(0)} \Phi\,d\mathcal{H}^1= \int_{l_{\overrightarrow{v_{i_0}}}\cap B_R(0)} \Phi\,d\mathcal{H}^1,
\]
contradiction.

Let us assume that $\gamma$ satisfies \textbf{item ii)}. In this case,  $\gamma_s^i\cap B_{4R}(0)$  has a connected component $\gamma^i$ intersecting $B_{2R}$ which converges in  $C^{1,\frac{1}{2}}(B_R(0)-\{0\})$ to   the lines $\gamma_A$ and $\gamma_B$ with multiplicity one. Moreover,  $\theta_s^i$ converge to a constant $\theta_0$ on each connected component of $\gamma^i \cap (B_R(0)-B_r(0))$. This implies that $\gamma_A=\gamma_B$ with the same orientation or the angle between   $\gamma_A$ and $\gamma_B$ is $\frac{\pi}{n}$. The first case cannot happen since $I_2(\beta_s^i, \mathbb{S}^1(0, r))= 0$, where $I_2(\cdot, \cdot)$ is the intersection number mod 2. The second case cannot happen since $\{\gamma_t\}\subset \Omega_{\frac{\pi}{n}- \delta_t}$ by Lemma \ref{strictly contained cone}. Hence, the origin  is not a singularity if we assume that $\gamma_T\neq \{0\}$.

On the other hand, no singularities away from the origin occur. Indeed, in \cite{O} J. Oaks complement the work of S. Angenent on singularities of equations of type $\frac{d}{dt}\gamma_t= V(\overrightarrow{T},k)\overrightarrow{N}$ by showing that near the singularity the curve  $\gamma_t$ must lose a self intersection. Since Proposition \ref{non-avoidance equivariant flow} asserts the only self intersection of $\gamma_t$ is at the origin we are done.   

Now let us prove that the tangent flow at the singular point is a line through the origin with multiplicity $2$.
For this we choose  a sequence of scale
factors $\lambda_i \rightarrow +\infty$  and
we set
$\gamma_{s}^i=\lambda_i\gamma_{T+\frac{s}{\lambda_i^2}}$ defined in $[-T\lambda_i^2,0)$. As discussed before $\gamma_s^i$ converges in $C_{\text{loc}}^{1,\frac{1}{2}}(\mathbb{R}^2-\{0\})$ to a union of  two lines through the origin for almost every $s$ fixed. Let us denote them by $l_A$ and $l_B$. As $\text{Area}(\gamma_t)$ is going to zero there exist a unique $t_i\in [0,T)$ for which $\text{Area}(\gamma_{t_i})=\frac{1}{\lambda_i^2}$. This implies that $\textrm{Area}(\gamma_{s_1^i}^i)=1$, where $s_1^i$ is given by $s_1^i=-\lambda_i^2(T-t_i)$. Since $\pi(T-t)\leq A(t)\leq 3\pi(T-t)$ by Lemma \ref{variation area}, we obtain that
 $s_1^i \in [-\frac{1}{\pi},-\frac{1}{3\pi}]$. In particular, if
$s^{*}=-\frac{1}{3\pi}$, then $\limsup_{i  \rightarrow \infty}\textrm{Area}(\gamma_{s^{*}}^i)\leq 1$. Therefore, $\gamma_{s^{*}}^i$  must converge to $2\gamma_A+ 2\gamma_B$ or $\gamma_A=\gamma_B$ since $\gamma_{s^{*}}^i$ is becoming non-compact enclosing bounded area. The first case does not happen as it violates the assumptions i) and ii) as discussed above.
\end{proof}

Let us construct equivariant Lagrangian spheres in $\mathbb{R}^4$ that do not collapse to a point along the mean curvature flow.

\begin{example}
	Let $\gamma_0$ be the curve $\gamma^{\alpha}(u)=\sin(\frac{\pi u}{\alpha})^{-\frac{\alpha}{\pi}}(\cos(u),\sin(u))$ with $u \in \mathbb{R}$. The  existence of a solution of the equivariant flow starting at $\gamma^{\alpha}$ is given in \cite{N}, let us denote it by $\{\gamma_t\}_{t\in[0,T_{\alpha})}$. It is shown in \cite{N} that when $\alpha>\frac{\pi}{2}$, then $T_{\alpha}<\infty$ and  $\gamma_t$ develops a singularity at the origin. When $\alpha \in (0,\pi)$, then  $\gamma^{\alpha}$  is contained in $\Omega_{\alpha}$ and it is asymptotic to its boundary. Consider the region $U_{\alpha}$ in $\Omega_{\alpha}$ that is bounded by $\{\gamma^{\alpha}\}\cup \{-\gamma^{\alpha}\}$. One can check that $U_{\alpha}$ has infinite area. Choose $\beta\in \mathcal{C}$ contained in $U_{\alpha}$   whose  area enclosed, $\text{Area}(\beta)$, is greater than $3\pi\, T_{\alpha}$. See Figure 2 for the case  $\alpha=\pi$. Let $\{\beta_t\}_{t\in[0,T)}$ be the solution of the equivariant flow starting at $\beta$. By the avoidance principle, $\beta_t$ and $\gamma_t$ do not intersect. Hence, $T<T_{\alpha}$. On the other hand, by Lemma \ref{variation area} we have that $\text{Area}(\beta_{T})\geq \text{Area}(\beta)- 3\pi T\geq 3\pi (T_{\alpha}-T)>0$. Therefore, a non trivial singularity must occur at the origin.
\end{example}
\begin{figure}[h]\label{whitneysphere2}
	\includegraphics[scale=0.15]{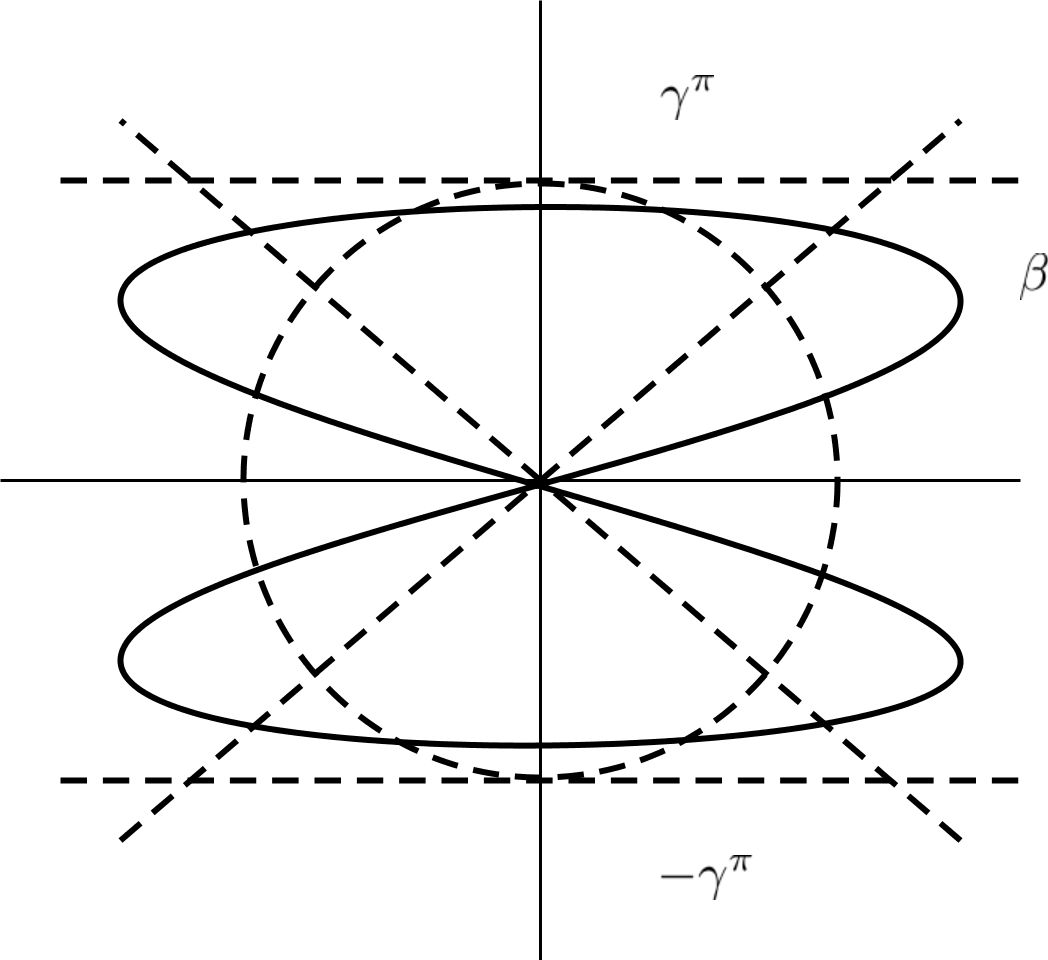}
	\caption{Curve $\beta$}
\end{figure}

Let us show that any Type II dilation of $\gamma_t$ near the singularity converges to an eternal solution of curve shortening flow. As in Chapter 4 in \cite{M}, there exist for each $k>0$, points $z_k\in \gamma_t(\mathbb{S}^1)$,  $t_k\in [0, T-\frac{1}{k}]$, and scaling $\lambda_k>0$ such that  $\beta_s^k= \lambda_k (\gamma_{T+\frac{s}{\lambda_k^2}}-z_k)$ satisfies 
\begin{eqnarray*}
	\frac{d}{ds}\beta_s^k= \overrightarrow{k}(\beta_s^k)- \frac{(\beta_s^k+ \lambda_k z_k)^{\perp}}{|\beta_s^k+ \lambda_k z_k|^2},
\end{eqnarray*}
where $s\in (a_k,b_k)$. Moreover, $\lim_{k\rightarrow\infty}a_k=-\infty$, $\lim_{k\rightarrow\infty}b_k=\infty$, and $0< \lim_{k\rightarrow\infty}\sup_{(a_k,b_k)\times \mathbb{S}^1}|\overrightarrow{k}(\beta_s^k)|\leq C$. It is proved that $\beta_s^k$ converge smoothly as $k\rightarrow 
\infty$ to a non-compact flow $(\beta_s)_{s\in \mathbb{R}}$. We claim that  $\lim_{k\rightarrow\infty} \lambda_k z_k = \infty$. If not,  then we could replace the points $z_k$ by $z=0$ and obtain the same conclusion. This is impossible since central dilations converge to lines. Therefore, as $k\rightarrow \infty$, 
\[\frac{d}{ds} \beta_s= \overrightarrow{k}(\beta_s).
\]

\noindent\textit{Acknowledgements.}
I would like to thank Jason Lotay for suggesting this problem and for his encouragement and support during this work.
I also thank my advisor Andr\'{e} Neves for many helpful conversations.
This work was supported by the Engineering and
Physical Sciences Research Council [EP/L015234/1], and the EPSRC Centre for Doctoral
Training in Geometry and Number Theory (London School of Geometry and
Number Theory), University College London.

\end{document}